\documentclass[12pt,a4paper]{article}
\usepackage{amsmath}
\usepackage{amsthm}
\usepackage{amsfonts}
\usepackage{amssymb}
\usepackage{stmaryrd}
\usepackage{latexsym}
\usepackage{eucal}
\usepackage[english]{babel}
\usepackage{float}
\usepackage{graphics}
\usepackage[usenames]{color}

\theoremstyle{theorem}
\newtheorem{theorem}{Theorem}[section]
\newtheorem{lemma}[theorem]{Lemma}
\newtheorem{corollary}[theorem]{Corollary}

\newtheorem{definition}[theorem]{Definition}

\theoremstyle{definition}

\newtheorem{example}[theorem]{Example}

\renewenvironment{proof}[1][\proofname]{\noindent \textbf{Proof. }}
{%
  \qed\endtrivlist
}

\author{A.V.  Scerbacova and V.A. Shcherbacov }
\title{About spectrum  of   $T_2$-quasigroups}

\begin{document}

\maketitle

\begin{abstract}
\noindent
We give information about some properties and spectrum of quasigroups with the following identity  $x(y \cdot yx) = y$.

\smallskip


\smallskip

\end{abstract}

\section{Introduction}
Definitions and elementary properties of quasigroups can be found in \cite{VD, 1a, HOP}.

V. D. Belousov \cite{BELPAR, BELPAR_05} by the study of orthogonality of quasigroup parastrophes  proved that there exists exactly seven   parastrophically   non-equivalent  identities which guarantee that a quasigroup is orthogonal to at least one its parastrophe:
\begin{align}
&  x(x \cdot xy) = y \, &&(C_3 \, \textrm{law}) \label{identity_$C_3$}\\
&  x(y \cdot yx) = y  \, &&\textrm{of type} \, T_2 \, \cite{BELPAR} \label{T_2} \\
&  x \cdot xy = yx \, &&\textrm{(Stein's 1st law)} \label{Steins 1st}\\
&  xy \cdot x = y \cdot xy \, &&\textrm{(Stein's 2nd law)} \label{identity_Stein 2nd}\\
&  xy \cdot yx = y \, &&\textrm{(Stein's 3nd law)} \label{identity_Stein's 3rd}\\
&  xy \cdot y = x \cdot xy \, &&\textrm{(Schroder's 1st law)} \label{identity_Schroder's 1st}\\
&  yx \cdot xy = y \, &&\textrm{(Schroder's 2nd law)}.  \label{Schroders 2nd}
\end{align}

The names of identities (\ref{Steins 1st})--(\ref{Schroders 2nd})  originate from Sade's paper \cite{AS_57}.
We follow \cite{BENNETT_89} in the calling  of identity (\ref{identity_$C_3$}).
All these identities can be  obtained in a unified way using criteria of orthogonality and quasigroup translations  \cite{MS05}. For example, identity (\ref{T_2}), that guarantee orthogonality of a quasigroup $(Q, \cdot)$ and its $(2\,3)$-parastrophe,  can be obtained from the following translation equality $L^2_y x = P_y x$.

A quasigroup $(Q, \cdot)$  with identity $x\cdot x = x$ is called idempotent. The set $\frak Q$ of natural numbers for which there exist quasigroups with a property $T$, for example, the property of idempotency, is called the spectrum of the property $T$ in the class of quasigroups. Often it is used the following phrase: spectrum  of quasigroups with a property $T$.  Therefore we can say that spectra of quasigroups with identities (\ref{Steins 1st})--(\ref{Schroders 2nd})  were studied in  \cite{CCL_80, BENNETT_89, Pelling_78, Pasha_14}.

\section{Medial $T_2$-quasigroups}

The problem of the study of $T_2$-quasigroups is posed in \cite{BELPAR, BELPAR_05}. In \cite{Pasha_09}
the following proposition (Proposition 7) is proved. We formulate this proposition in slightly changed form.

\begin{theorem}
 If a $T_2$-quasigroup $(Q, \cdot)$  is isotopic to an abelian
group $(Q, \oplus)$, then for every element $b\in Q$  there exists an isomorphic copy $(Q, +)  \cong (Q, \oplus)$
such that $x \cdot  y = IL^3_b(x) + L_b(y) + b$, for all $x,  y \in  Q$, where $x + Ix = 0$ for all $x\in Q$.
\end{theorem}

\begin{definition}
A quasigroup $(Q, \cdot)$ of the form $x\cdot y = \varphi x + \psi y + b$, where $(Q, +)$ is an abelian group, $\varphi, \psi$ are automorphisms of the group $(Q, +)$,  $b$ is a fixed element of the set $Q$ is called $T$-quasigroup. If, additionally, $\varphi \psi = \psi \varphi$, then  $(Q, \cdot)$ is called  medial quasigroup \cite{pntk, VD, HOP, 1a}.
\end{definition}

\begin{theorem} \label{T_2-T_quas}
A T-quasigroup $(Q, \cdot)$ of the form
\begin{equation} \label{Form_of_T_quasigroup}
x\cdot y = \varphi x + \psi y + b
\end{equation}
 satisfies  $T_2$-identity if and only if $\varphi = I\psi^3$, $\psi^5   + \psi^4   + 1 = 0$, where $1$ is identity automorphism of the group $(Q, +)$ and $0$ is zero endomorphism of this group, $\psi^2 b + \psi b + b = 0$.
\end{theorem}
\begin{proof}
We rewrite $T_2$-identity using the right part of the form (\ref{Form_of_T_quasigroup}) as follows:

\begin{equation} \label{Equality_9}
 \varphi x + \psi (\varphi y + \psi (\varphi y + \psi x + b)+ b) + b = y\\
\end{equation}
or, taking into consideration that $(Q,+)$ is an abelian group, $\varphi, \psi$ are its automorphisms after simplification of equality  (\ref{Equality_9}) we have
\begin{equation} \label{Equality_95}
 \varphi x +  \psi \varphi y + \psi^2 \varphi y + \psi^3  x  + \psi^2 b + \psi b + b = y.\\
\end{equation}
If we put in the equality (\ref{Equality_95}) $x= y =0$,
then we obtain
\begin{equation} \label{Equality_9125}
\psi^2 b + \psi b + b = 0, \\
\end{equation}
where $0$ is the identity (neutral) element
of the group $(Q, +)$.

Therefore we can rewrite equality (\ref{Equality_95}) in the following form
\begin{equation} \label{Equality_951}
 \varphi x +  \psi \varphi y + \psi^2 \varphi y + \psi^3  x  = y.\\
\end{equation}
If  we put in the equality (\ref{Equality_951}) $y=0$, then we obtain  that  $\varphi x +   \psi^3  x = 0$. Therefore $\varphi = I\psi^3$, where, as and above, $x + Ix = 0$ for all $x\in Q$.

Notice, in any abelian group $(Q, +)$ the map $I$ is an automorphism
 of this group. Really, $I(x+y) = Iy+Ix= Ix+Iy $.

 Moreover,  $I\alpha = \alpha I$ for any automorphism of the group $(Q, +)$. Indeed, $\alpha x + I\alpha x = 0$. From the other side $\alpha x + \alpha I x = \alpha (x + I x) = \alpha 0 =0$. Comparing the left sides we have $\alpha x + I\alpha x = \alpha x + \alpha I x $, $I\alpha x = \alpha I x $, i.e., $\alpha I = I\alpha$.

It is well known that $I^2 = \varepsilon$, i.e., $-(-x)= x$. Indeed, from equality $x+Ix =0$ using commutativity we have $Ix+ x =0$. From the other side $I(x+Ix)=0$, $Ix+I^2x=0$. Then $Ix+ x = Ix+I^2x$, $x= I^2x$ for all $x\in Q$.

If we put in the equality (\ref{Equality_951}) $x=0$, then we obtain  that
\begin{equation}\label{Equality_954}
 \psi \varphi y + \psi^2 \varphi y = y.
\end{equation}
 If we  substitute in the equality (\ref{Equality_954}) expression  $I\psi^3$ instead of $\varphi$, then we have
$I \psi^5  y + I \psi^4  y = y$, $\psi^5  y + \psi^4  y = Iy$,  $\psi^5  y + \psi^4  y + y = 0$. The last condition can be written in the form
$\psi^5   + \psi^4   + 1 = 0$, where $1$ is identity automorphism of the group $(Q, +)$ and $0$ is zero endomorphism of this group.

Converse. If we take into consideration  that $\psi^2 b + \psi b + b = 0$, then from  equality (\ref{Equality_95}) we obtain equality (\ref{Equality_951}). If we substitute in equality  (\ref{Equality_951})  the following equality $\varphi = I\psi^3$,
then we obtain  $\psi I\psi^3  y + \psi^2 I\psi^3 y  = y$, $\psi^4 I  y + \psi^5 I y  = y$ which is equivalent to the  equality $\psi^5 y + \psi^4 y + y = 0$. Therefore $T$-quasigroup $(Q, \cdot)$ is $T_2$-quasigroup.
\end{proof}
\begin{corollary}
Any $T$-$T_2$-quasigroup is medial.
\end{corollary}
\begin{proof}
The proof follows from equality $\varphi = I\psi^3$ (see Theorem \ref{T_2-T_quas}).
\end{proof}

\begin{corollary} \label{T_2-T_quas_}
A T-quasigroup $(Q, \cdot)$ of the form
\begin{equation} \label{Form_of_T_quasigroup+}
x\cdot y = \varphi x + \psi y
\end{equation}
 satisfies  $T_2$-identity if and only if $\varphi = I\psi^3$,   $\psi^5  + \psi^4  + 1 = 0$.
\end{corollary}
\begin{proof}
It is easy to see.
\end{proof}

\section{$T_2$-quasigroups from the rings of residues}

We  use rings of residues modulo $n$, say $(R, +, \cdot, 1)$,  and Theorem \ref{T_2-T_quas} to construct $T_2$-quasigroups. Here $(R, +)$ is cyclic group of order $n$, i.e., it is the group $(Z_n, +)$ with the generator element $1$. It is clear that in many cases the element $1$ is not a unique generator element, $(R, \cdot)$ is a commutative semigroup \cite{MALTSEV}.

Multiplication of an element  $b \in R$  on all  elements of the group $(R, +)$ induces an  endomorphism of the group $(R, +)$, i.e., $b\cdot(x+y) = b\cdot x + b\cdot y$. If $g.c.d. (b, n) = 1$, then the element $b$ induces an automorphism of the group $(R, +)$ and it is called an invertible element of the ring $(R, +, \cdot, 1)$.

Next theorem is a specification of Theorem  \ref{T_2-T_quas} on medial $T_2$-quasigroups defined using rings of residues modulo $n$.
We denote by the symbol  $\mathbb{Z}$  the set of integers,  we denote by  $|n|$  module of the number $n$.

\begin{theorem} \label{dividers_TH}
Let $(Z_r, +, \cdot, 1)$ be a ring of residues modulo $r$ such that $f(k) =(k^5+k^4+1)\equiv 0 \pmod{r}$ for some $k\in \mathbb Z$. If  $g.c.d. (|k|, r) =1$, $k^2 \cdot b + k\cdot b + b \equiv 0\pmod{r}$ for some $b\in Z_r$,  then there exists $T_2$-quasigroup $(Z_{r}, \circ)$ of the form $x\circ  y = - k^3 \cdot x + k \cdot y + b$.
\end{theorem}
\begin{proof}
We can use Theorem \ref{T_2-T_quas}.  The fact that $g.c.d. (|k|, r) =1$ guarantee that the multiplication on the number $k$ induces an automorphism of the group $(Z_{r},+)$. In this case the map $-k^3$ is also  a permutation as a product of permutations.
\end{proof}

\begin{example} \label{Sostavnoe_Chislo}
 Let $k = -3$. Then $f(-3) = (-3)^5 + (-3)^4 + 1 = -161= -(7)\cdot(23)$.
Therefore $-161 \equiv 0 \pmod{7}$ and $-161 \equiv 0 \pmod{23}$ and we have theoretical possibility to construct $T_2$ quasigroups of order $7, \, 23, 161$.

Case 1.
Let $r=7$. Then $k = -3 = 4 \pmod{7}$. In this case $-(k^3) =  -(-3)^3 =  27= 6 \pmod{7}$. It is clear that the elements $6$ and $4$ are invertible elements of the ring $(Z_{7}, +, \cdot, 1)$. Therefore quasigroup $(Z_{7}, \ast)$  with the form $x\ast y = 6\cdot x + 4\cdot y$ is $T_2$-quasigroup of order $7$.

Check. We have $6x+4(6y+4(6y+4x))=y$, $70x+24 y + 96 y  = y$, $y= y$, since $70\equiv 0 \pmod{7}, 120\equiv 1 \pmod{7}$.

In order to construct $T_2$-quasigroups over the ring $(Z_7, +, \cdot, 1)$ with non-zero element $b$ we must  solve congruence
$(-3)^2 \cdot b + (-3)\cdot b + b \equiv 0\pmod{7}$. We have $7 \cdot b \equiv 0\pmod{7}$. The last equation is true for any possible value of the element $b$. Therefore the following quasigroups are $T_2$-quasigroups of order $7$: $x\circ y = 6\cdot x + 4\cdot y + i$, for any $i\in \{1, 2, \dots, 5, 6 \}$.

Case 2.
Let $r = 23$. Then $k = -3 = 20 \pmod{23}$. In this case $-(k^3) = -(-3)^3 = 27=4 \pmod{23}$. It is clear that the elements $20$ and $4$ are invertible elements of the ring $(Z_{23}, +, \cdot, 1)$. Therefore quasigroup $(Z_{23}, \ast)$  with the form $x\ast y = 4\cdot x + 20\cdot y$ is $T_2$-quasigroup of order $23$.

Check. We have $4x+20(4y+20(4y+20x))=y$, $4x+80y +1600 y + 8000 x = y$, $y=y$, since $8004\equiv 0 \pmod{23}, 1680\equiv 1 \pmod{23}$.

In order to construct $T_2$-quasigroups over the ring $(Z_{23}, +, \cdot, 1)$ with non-zero element $b$ we must  solve congruence
$(-3)^2 \cdot b + (-3)\cdot b + b \equiv 0\pmod{23}$. We have $7 \cdot b \equiv 0\pmod{23}$. This congruence modulo has unique solution $b\equiv 0 \bmod{23}$, since $g.c.d. (7, 23) = 1$.

Case 3.
Let $r=161$. Then $k = -3 = 158 \pmod{161}$. Recall the number $161$ is not prime. In this case $-(k)^3 = -(-3)^3 = 27 \pmod{161}$, $g.c.d. (27, 161) = 1$, the elements $158$ and $27$ are invertible elements of the ring $(Z_{161}, +, \cdot, 1)$. Therefore quasigroup $(Z_{161}, \circ)$  with the form $x\circ y = 27\cdot x + 158\cdot y$ is medial $T_2$-quasigroup of order $161$.

Check. $27 x+ 4266 y + 674028 y + 3944312 x = y$, $y=y$, since $3944339 \equiv 0 \pmod{161}$, $678294\equiv 1 \pmod{161}$.

In order to construct $T_2$-quasigroups over the ring $(Z_7, +, \cdot, 1)$ with non-zero element $b$ we must  solve congruence
 $7 \cdot b \equiv 0\pmod{161}$.  It is clear that $g.c.d. (7, 161) = 7$. Therefore this congruence has 6 non-zero solutions, namely, $b \in \{23, 46, 69, 92, 115, 138\} = D$.

 The following quasigroups are $T_2$-quasigroups of order $161$: $x\circ y = 27\cdot x + 158\cdot y + i$, for any $i\in D$.
\end{example}

\begin{example}\label{Dvadtsat_Otrits}
We list some values of the polynomial $f$:

\begin{align*}
& f(-20) = -3039999, f(-19) = -2345777, f(-18) = -1784591, f(-17) = -1336335, \\
& f(-16) = -983039, f(-15) = -708749, f(-14) = -499407,  f(-13) = -342731,   \\
& f(-12) = -228095, f(-11) = -146409, f(-10) = -89999, f(-9) = -52487,  \\
& f(-8) = -28671, f(-7) = -14405, f(-6) = -6479, f(-5) = -2499, f(-4) = -767, \\
&  f(-3) = -161, f(-2) = -15,  f(-1) = 1,  f(1) =3, f(2) =49, \\
& f(3)=325, f(4) = 1281, f(5) =3751, f(6) =9073, f(7) =19209, \\
&  f(8) =36865, f(9) =65611, f(10) =110001, f(11) =175693, f(12) =269569, \\
&  f(13) =399855, f(14) =576241, f(15) =810001, f(12) =269569, \\
&  f(17) =1503379,  f(18) =1994545,  f(19) =2606421, f(20) =3360001.
\end{align*}

The set of prime divisors  of the numbers of the set $\{f(-20), f(-19), \dots, f(-1), $ $ f(1), \dots, f(20) \}$ contains  the following primes
\begin{align*}
& \{ 3, 5, 7, 13, 19, 23, 37, 43, 59, 61, 73, 101, 157, 211, 241, 307, 341, 347, 421,  503, 719, 833, 977,\\
& 979, 1163, 1319, 2183, 2881, 3359, 3751, 5047,  5813, 6403, 7373,  9073,  10033, 25099,  36667, 166469, \\
& 269569, 868807, 1503379 \}.
\end{align*}
\end{example}
We can use  presented  numbers for  construction of  $T_2$-quasigroups over the rings of residues.

In order to construct $T_2$-quasigroups it is possible to use direct products of $T_2$-quasigroups. It is clear  that direct product of  $T_2$-quasigroups is a $T_2$-quasigroup.

It is possible to use and the following arguments. Class of $T_2$ quasigroups  is defined using $T_2$-identity and it forms a variety in signature with three binary operations, namely, with the operations $\cdot$, $\slash$, and $\backslash$ \cite{MALTSEV}. It is known that any variety is closed relatively to the operator  of the direct product \cite{MALTSEV}.

\section{Examples of $T_2$-quasigroups} \label{Section_Examples of $T_2$}

Using Mace4 \cite{MAC_CUNE_MACE} we construct the following examples of $T_2$-quasigroups.
\[
\begin{array}{lll}
\begin{tabular}{r|rrr}
$\ast$ & 0 & 1 & 2\\
\hline
    0 & 0 & 1 & 2 \\
    1 & 2 & 0 & 1 \\
    2 & 1 & 2 & 0
\end{tabular}
&
\begin{tabular}{r|rrrrr}
$\circ$ & 0 & 1 & 2 & 3 & 4\\
\hline
    0 & 0 & 2 & 4 & 1 & 3 \\
    1 & 2 & 1 & 3 & 4 & 0 \\
    2 & 4 & 3 & 2 & 0 & 1 \\
    3 & 1 & 4 & 0 & 3 & 2 \\
    4 & 3 & 0 & 1 & 2 & 4
\end{tabular}
&
\begin{tabular}{r|rrrrrrr}
$\star$ & 0 & 1 & 2 & 3 & 4 & 5 & 6\\
\hline
    0 & 0 & 2 & 3 & 1 & 6 & 4 & 5 \\
    1 & 6 & 1 & 4 & 0 & 5 & 2 & 3 \\
    2 & 5 & 0 & 2 & 6 & 3 & 1 & 4 \\
    3 & 4 & 5 & 0 & 3 & 2 & 6 & 1 \\
    4 & 3 & 6 & 1 & 5 & 4 & 0 & 2 \\
    5 & 2 & 3 & 6 & 4 & 1 & 5 & 0 \\
    6 & 1 & 4 & 5 & 2 & 0 & 3 & 6
\end{tabular}
\end{array}
\]

\[
\begin{array}{ll}
\begin{tabular}{r|rrrrrrrr}
$\diamond$ & 0 & 1 & 2 & 3 & 4 & 5 & 6 & 7\\
\hline
    0 & 0 & 2 & 7 & 1 & 5 & 6 & 3 & 4 \\
    1 & 4 & 1 & 6 & 2 & 3 & 7 & 5 & 0 \\
    2 & 5 & 7 & 2 & 6 & 0 & 1 & 4 & 3 \\
    3 & 7 & 5 & 0 & 3 & 2 & 4 & 1 & 6 \\
    4 & 3 & 6 & 1 & 7 & 4 & 2 & 0 & 5 \\
    5 & 1 & 4 & 3 & 0 & 6 & 5 & 7 & 2 \\
    6 & 2 & 0 & 5 & 4 & 7 & 3 & 6 & 1 \\
    7 & 6 & 3 & 4 & 5 & 1 & 0 & 2 & 7
\end{tabular}
&
\begin{tabular}{r|rrrrrrrrrrr}
$\bullet$ & 0 & 1 & 2 & 3 & 4 & 5 & 6 & 7 & 8 & 9 & 10\\
\hline
    0 & 0 & 2 & 7 & 1 & 6 & 9 & 10 & 8 & 3 & 4 & 5 \\
    1 & 10 & 1 & 9 & 6 & 8 & 4 & 7 & 0 & 2 & 5 & 3 \\
    2 & 5 & 10 & 2 & 7 & 9 & 1 & 3 & 4 & 0 & 6 & 8 \\
    3 & 6 & 5 & 0 & 3 & 2 & 10 & 8 & 1 & 4 & 7 & 9 \\
    4 & 3 & 7 & 1 & 9 & 4 & 6 & 2 & 5 & 10 & 8 & 0 \\
    5 & 7 & 6 & 4 & 8 & 0 & 5 & 9 & 10 & 1 & 3 & 2 \\
    6 & 1 & 4 & 8 & 5 & 3 & 0 & 6 & 2 & 9 & 10 & 7 \\
    7 & 9 & 8 & 5 & 0 & 10 & 3 & 4 & 7 & 6 & 2 & 1 \\
    8 & 4 & 0 & 3 & 10 & 7 & 2 & 5 & 9 & 8 & 1 & 6 \\
    9 & 8 & 3 & 10 & 2 & 1 & 7 & 0 & 6 & 5 & 9 & 4 \\
    10 & 2 & 9 & 6 & 4 & 5 & 8 & 1 & 3 & 7 & 0 & 10
\end{tabular}
\end{array}
\]

\bigskip

\[
\begin{array}{lcr}
\begin{tabular}{r|rrrr}
$\boxtimes$ & 0 & 1 & 2 & 3\\
\hline
    0 & 0 & 2 & 3 & 1 \\
    1 & 1 & 3 & 2 & 0 \\
    2 & 2 & 0 & 1 & 3 \\
    3 & 3 & 1 & 0 & 2
\end{tabular}
&
\begin{tabular}{r|rrrrrrr}
$\boxdot$ & 0 & 1 & 2 & 3 & 4 & 5 & 6\\
\hline
    0 & 1 & 2 & 0 & 5 & 4 & 6 & 3 \\
    1 & 0 & 3 & 6 & 4 & 1 & 2 & 5 \\
    2 & 5 & 0 & 4 & 2 & 3 & 1 & 6 \\
    3 & 3 & 1 & 5 & 6 & 2 & 4 & 0 \\
    4 & 6 & 5 & 2 & 1 & 0 & 3 & 4 \\
    5 & 2 & 4 & 3 & 0 & 6 & 5 & 1 \\
    6 & 4 & 6 & 1 & 3 & 5 & 0 & 2
\end{tabular}
&
\begin{tabular}{r|rrrrrrrrr}
$\boxplus$ & 0 & 1 & 2 & 3 & 4 & 5 & 6 & 7 & 8\\
\hline
    0 & 1 & 2 & 0 & 3 & 4 & 5 & 6 & 7 & 8 \\
    1 & 0 & 1 & 2 & 8 & 6 & 4 & 5 & 3 & 7 \\
    2 & 2 & 0 & 1 & 7 & 5 & 6 & 4 & 8 & 3 \\
    3 & 4 & 5 & 6 & 0 & 3 & 7 & 8 & 1 & 2 \\
    4 & 3 & 7 & 8 & 4 & 0 & 1 & 2 & 5 & 6 \\
    5 & 8 & 3 & 7 & 6 & 2 & 0 & 1 & 4 & 5 \\
    6 & 7 & 8 & 3 & 5 & 1 & 2 & 0 & 6 & 4 \\
    7 & 6 & 4 & 5 & 2 & 8 & 3 & 7 & 0 & 1 \\
    8 & 5 & 6 & 4 & 1 & 7 & 8 & 3 & 2 & 0
\end{tabular}
\end{array}
\]

\[
\begin{tabular}{r|rrrrrrrrrrr}
$\boxminus$ & 0 & 1 & 2 & 3 & 4 & 5 & 6 & 7 & 8 & 9 & 10\\
\hline
    0 & 0 & 2 & 4 & 1 & 5 & 3 & 7 & 9 & 6 & 10 & 8 \\
    1 & 8 & 1 & 7 & 10 & 0 & 2 & 9 & 6 & 3 & 5 & 4 \\
    2 & 6 & 8 & 2 & 4 & 9 & 0 & 1 & 10 & 5 & 7 & 3 \\
    3 & 10 & 6 & 0 & 3 & 1 & 9 & 8 & 4 & 7 & 2 & 5 \\
    4 & 7 & 5 & 6 & 0 & 4 & 10 & 3 & 2 & 1 & 8 & 9 \\
    5 & 9 & 0 & 3 & 8 & 7 & 5 & 2 & 1 & 10 & 4 & 6 \\
    6 & 4 & 10 & 9 & 5 & 8 & 1 & 6 & 3 & 2 & 0 & 7 \\
    7 & 5 & 3 & 8 & 2 & 10 & 6 & 4 & 7 & 9 & 1 & 0 \\
    8 & 2 & 7 & 10 & 9 & 3 & 4 & 5 & 0 & 8 & 6 & 1 \\
    9 & 3 & 4 & 1 & 7 & 6 & 8 & 10 & 5 & 0 & 9 & 2 \\
    10 & 1 & 9 & 5 & 6 & 2 & 7 & 0 & 8 & 4 & 3 & 10
\end{tabular}
\]

\begin{lemma}
 There exist $T_2$-quasigroups of order $2^{\,k}$ for any  $k\geq 2$.
\end{lemma}
\begin{proof}
Since $T_2$-quasigroup with the operation $\boxtimes$
has the order $2^2$ and $T_2$-quasigroup with the operation $\diamond$ has the order $2^3$,  $g.c.d. (2,3) = 1$.
\end{proof}

\begin{example}
There exist $T_2$-quasigroup of order $2^{11}$ since $11 = 2\cdot 1 + 3\cdot 3$.
\end{example}

It is easy to check that does not exist $T_2$-quasigroups of order 2.

\section{Spectra of idempotent $T_2$-quasigroups}

\begin{definition}
A pairwise balanced design (or PBD) is a set $X$ together with a family of subsets of $X$ (which need not have the same size and may contain repeats) such that every pair of distinct elements of $X$ is contained in exactly $\lambda$ (a positive integer) subsets. The set $X$ is allowed to be one of the subsets, and if all the subsets are copies of $X$, the PBD is called trivial. The size of $X$ is $v$ and the number of subsets in the family (counted with multiplicity) is $b$ \cite{WIKI_16}. \index{PBD}
\end{definition}

Theorem \ref{Wilson Theorem} was proved by  Richard M. Wilson. \index{theorem!Wilson R.M.}

\begin{theorem}  \label{Wilson Theorem}
Let $K =\{ k_1, k_2, \dots \}$  be a set of positive integers. Let
$\alpha (K) = g.c.d. (k(k - 1) \, | \,  k \in  K)$ and $\beta (K) = g.c.d. (k - 1 \, | \, k \in K) $.
Then there is an integer $v_0$  such that $v\geq v_0$ and if $v(v-1)\equiv 0 \pmod{\alpha (K)}$ and
$v-1\equiv 0 \pmod{\beta(K)}$, then there is a pairwise balanced block design with $\lambda  = 1$,
order $v$ and block sizes $k_1, k_2, \dots$ \cite{WILSON_Rich, CCL_80}.
\end{theorem}

The following theorem establishes a connection between some quasigroups and block designs.
\begin{theorem} \label{Models_Of_Quas}
Let $V$ be a variety (more generally universal class) of algebras which is idempotent and which is based on
$2$-variable identities. If there are models of $V$ of orders $k_1, k_2, \dots$  and if there is a
pairwise balanced block design with $\lambda  = 1$, order $v$, and blocks of sizes $k_1, k_2, \dots$,
then there is a model of $V$ of order $v$ \cite{CCL_80}.
\end{theorem}
See also \cite{STEIN_64}[Section 4], \cite{Pelling_78}.

\begin{theorem} \label{Idempot_T2_quas}
Idempotent $T_2$-quasigroups there exist for any big $v$ such that $v \geq v_0$.
\end{theorem}
\begin{proof}
The proof is based on  Theorems  \ref{Wilson Theorem} and  \ref{Models_Of_Quas}. From Section \ref{Section_Examples of $T_2$}
it follows that there exist idempotent $T_2$-quasigroups of orders 5, 7, 8, 11.
Therefore $\alpha (K) =g.c.d.(20, 42, 56, 110) = 2$, $\beta (K) =g.c.d.(4, 6, 7, 10) = 1$.

Then  congruences $v(v-1)\equiv 0 \pmod{\alpha (K)}$ and $v-1\equiv 0 \pmod{\beta(K)}$ take the form
$v(v-1)\equiv 0 \pmod{2}$ and $v-1\equiv 0 \pmod{1}$. Easy to see that any positive  $v\geq 2$ satisfies these congruences.

Using Theorems  \ref{Wilson Theorem} and  \ref{Models_Of_Quas} we conclude that idempotent $T_2$-quasigroups there exist for any big $v$ such that $v \geq v_0$.
\end{proof}

The problem of the  estimation  of the number $v_0$ from Theorem  \ref{Idempot_T2_quas} is important.
For comparison it is known  that for Stein's quasigroups (quasigroups with identity (\ref{Steins 1st})) $v_0\leq 1042$  \cite{Pelling_78}.

\noindent \footnotesize{
A.V. Scerbacova\\
Gubkin Russian State Oil and Gas University \\
119991, Moscow, Leninsky Prospect, 65 \\
 Russia \\
E-mail: \emph{ E-mail: 	\emph{scerbik33@yandex.ru} }}

\medskip

\noindent \footnotesize{
V.A. Shcherbacov\\
Institute of Mathematics and Computer Science \\
Academy of Sciences of Moldova\\
MD-2028, str. Academiei, 5, Chisinau \\
Moldova \\
E-mail: \emph{ E-mail: 	\emph{scerb@math.md}}}

\end{document}